\documentclass[11pt]{article}
\usepackage{indentfirst}
\usepackage{latexsym}
\usepackage{hyperref}
\usepackage{graphicx}
\usepackage{mathrsfs}
\usepackage{bookmark}
\usepackage{enumerate}
\usepackage{amsfonts,amsthm,amssymb,mathtools}
\usepackage{amsmath}
\usepackage{latexsym, euscript, epic, eepic, color}
\usepackage{amssymb}

\usepackage{enumerate}
\usepackage[all]{xy}
\usepackage{setspace}
\allowdisplaybreaks

\usepackage{epstopdf}
\numberwithin{equation}{section}

 \allowdisplaybreaks

\begin{document}
\baselineskip=16pt

\title{Kirchhoff index, multiplicative degree-Kirchhoff index and spanning trees of the linear crossed polyomino chains\footnote{This work was supported by the National Natural Foundation of China [61773020]. Corresponding author: Yingui Pan(panygui@163.com)
 }
}

\author{Yingui Pan$^{a}$,   Jianping Li$^{a}$   \\
	\small  $^{a}$College of Liberal Arts and Sciences, National University of Defense Technology, \\
	\small  Changsha, China, 410073.\\
%\small $^{b}$  School of Mathematics and Computational Science, Hunan First Normal University,\\
%\small  Changsha, China, 410205.\\
%\small $^{c}$  School of Mathematics and Statistics, Shandong Normal University,\\
%\small Jinan, Shandong, China,  250014.
}

\date{\today}

\maketitle

\begin{abstract}
Let $G_n$ be a linear crossed polyomino chain with $n$ four-order complete graphs. In this paper, explicit formulas for the Kirchhoff index, the multiplicative degree-Kirchhoff index and the number of spanning trees of $G_n$ are determined, respectively. It is interesting to find that the Kirchhoff (resp. multiplicative degree-Kirchhoff) index of $G_n$ is approximately one quarter of its Wiener (resp. Gutman) index. More generally, let $\mathcal{G}^r_n$ be the set of subgraphs obtained by deleting $r$ vertical edges of $G_n$, where $0\leqslant r\leqslant n+1$. For any graph $G^r_n\in \mathcal{G}^r_{n}$, its Kirchhoff index and  number of spanning trees are completely determined, respectively. Finally, we show that the Kirchhoff index of $G^r_n$ is approximately one quarter of its Wiener index. 
\end{abstract}

%AMS Subject Classification  05C50;

%$V(G)=\{v_1,v_2,\ldots,v_n\}$

\section{Introduction}
Let $G=(V_G,E_G)$ be a simple connected graph with vertex set $V_G=\{v_1,v_2,\ldots,v_n\}$ and edge set $E_G$. The adjacent matrix $A(G)=(a_{ij})_{n\times n}$ is a $(0,1)$-matrix  such that $a_{ij}=1$ if and only if vertices $v_i$ and $v_j$ are adjacent. Let $D(G)=\textrm{diag}(d_{1},d_{2},\ldots,d_{n})$ be the diagonal matrix of vertex degrees, where $d_{i}$ is the degree of $v_i$  in $G$ for $1\leqslant i\leqslant n$. Then the  Laplacian matrix of $G$ is defined as ${L}(G)=D(G)-A(G)$.

The traditional distance between vertices $v_i$ and $v_j$, denoted by $d_{ij}$, is the length of a shortest path connecting them. Distance is an important invariant in graph theory and derives many distance-based invariants. One famous distance-based graph invariant of $G$ is the Wiener index \cite{020}, $W(G)$, which is the sum of distances between pairs of vertices in $G$, namely $W(G)=\sum_{i<j}d_{ij}$. The Gutman index of $G$ was defined as $\textrm{Gut}(G)=\sum_{i<j}d_id_jd_{ij}$ by Gutman in \cite{000e}. When $G$ is a tree of order $n$, he \cite{000e} showed that $\textrm{Gut}(G)=4W(G)-(2n-1)(n-1)$.

Based on the electronic network theory, Klein and Randi$\acute{c}$\cite{013}  introduced a new distance-based parameter, namely the resistance distance, on a graph. The resistance distance $r_{ij}$ is the effective resistance between vertices $v_i$ and $v_j$ when one puts one unit resistor on every edge of a graph $G$. This parameter is intrinsic to the graph and has many  applications in theoretical chemistry.  As an analogue to the Wiener index, one may define  $K\!f(G)=\sum_{i<j}r_{ij}$, which is known as the Kirchhoff index \cite{013} of $G$. Klein and Randi$\acute{c}$  \cite{013} found that $K\!f(G)\leqslant{W}(G)$ with equality if and only if $G$ is a tree. For an $n$-vertex connected graph $G$, Klein \cite{011} and Lov$\acute{a}$sz \cite{015}, independently, obtained that 
\begin{flalign}\label{eq11}
	&\hspace{1cm}{K\!f}(G)=n\sum^{n}_{i=2}\frac{1}{\mu_i}.&
\end{flalign}	
where $0=\mu_1<\mu_2\leqslant\cdots\leqslant\mu_n$ $(n\geqslant2)$ are the eigenvalues of $L(G)$.

Recently, the normalized Laplacian has attracted increasing attention because some results which were only known for regular graphs can be spread to all graphs. The normalized Laplacian matrix of $G$ is defined to be  $\mathcal{L}(G)=D(G)^{-\frac{1}{2}}L(G)D(G)^{-\frac{1}{2}}$. It should be stressed that ${(d_{i})}^{-\frac{1}{2}}=0$ for the degree of vertex $v_i$ in $G$ is $0$  \cite{006}. Therefore, one can easily verify that
\begin{flalign}
	&\hspace{1cm}(\mathcal{L}(G))_{ij}=\left\{ 
	\begin{array}{ll}
		{1,} &\textrm{if $i=j$};\\
		{-\frac{1}{\sqrt{d_{i}d_{j}}}}, &\textrm{if $i\ne j$ and $v_i\sim v_j$};\\
		{0,} & \textrm{otherwise}.\\
	\end{array} 
	\right. &
\end{flalign}

In 2007, Chen and Zhang \cite{005}  proposed a new resistance distance-based graph invariant, defined by $K\!f^{*}(G)=\sum_{i<j}d_{i}d_{j}r_{ij}$, which is called the multiplicative degree-Kirchhoff index (see  \cite{007,008}. This index is closely related to the spectrum of the normalized Laplacian matrix $\mathcal{L}(G)$. For an $n$-vertex connected graph $G$ with $m$ edges, Chen and Zhang \cite{005} showed that
\begin{flalign}\label{eq111}
	&\hspace{1cm}{K\!f}^{*}(G)=2m\sum^{n}_{i=2}\frac{1}{\lambda_i}.&
\end{flalign}	
where $0=\lambda_1<\lambda_2\leqslant\cdots\leqslant\lambda_n$ $(n\geqslant2)$ are the eigenvalues of $\mathcal{L}(G)$.	

Up to now, closed formulas for the  Kirchhoff index and the multiplicative degree-Kirchhoff index have been given for some linear chains, such as cycles \cite{012}, ladder graphs \cite{cin}, ladder-like chains \cite{004}, liner phenylenes \cite{peng,zhu}, linear polyomino chains \cite{010,022}, linear pentagonal chains \cite{he,ywang}, linear hexagonal chains \cite{009,022} and linear crossed hexagonal chains \cite{pan}. Some other topics on the Kirchhoff index  and the multiplicative degree-Kirchhoff index of a graph  may be referred to \cite{001,002,003,cle1,cle2,014,016,017,018,019,021,023,024,025,026} and references therein.

\begin{figure}[htbp]
	\centering 
	\includegraphics[height=2.      in, width=3.7   in,angle=0]{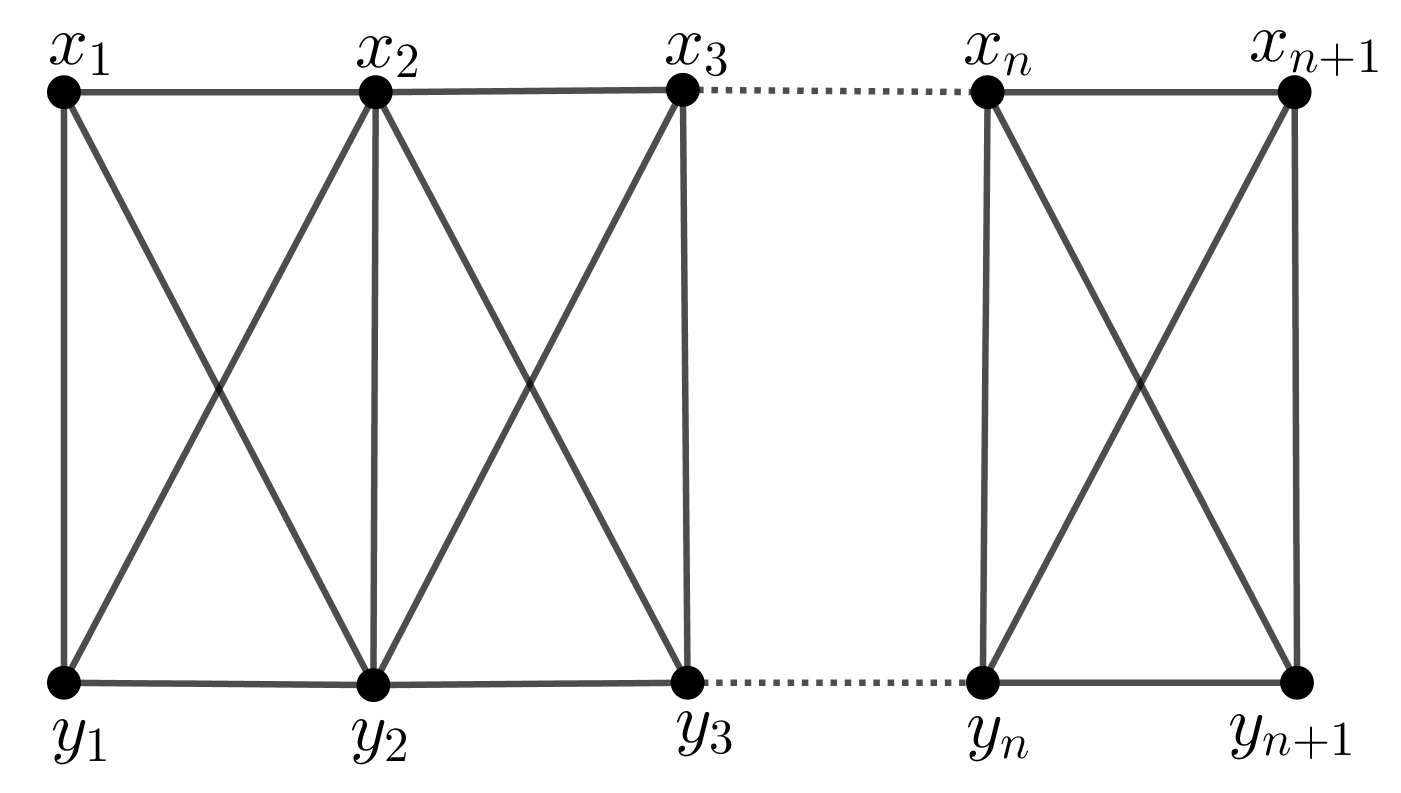}
	\caption{The linear crossed polyomino chain $G_n$ with some labelled vertices. }
	\label{fig01}
\end{figure}
Let $G_n$ be a linear crossed polyomino chain with $n$ four-order complete graphs as depicted in Fig. 1. Then it is routine to check that $|V_{G_n}|=2n+2$ and $|E_{G_n}|=5n+1$. Let $E^{\prime}$ be the set of vertical edges of $G_n$, where $E^{\prime}=\{ii^{\prime}:i=1,2,\dots,n+1\}$.     Let $\mathcal{G}^r_n$ be the set of subgraphs obtained by deleting $r$ vertical edges from $G_n$, where $0\leqslant r\leqslant n+1$. Obviously, $\mathcal{G}^0_n=\{G_n\}$.

In this paper, explicit formulas for the Kirchhoff index, the multiplicative degree-Kirchhoff index and the number of spanning trees of $G_n$ are determined, respectively. We are surprised to find that the Kirchhoff (resp. multiplicative degree-Kirchhoff) index of $G_n$ is approximately one quarter of its Wiener (resp. Gutman) index. More generally, for any graph $G^r_n\in \mathcal{G}^r_{n}$,  its Kirchhoff index and  number of spanning trees are completely determined, respectively. Furthermore, the Kirchhoff index of $G^r_n$ is found to be approximately one quarter of its Wiener index.

\section{Preliminaries}
In this paper, we denote by $\Phi(B)=\det(xI-B)$ the characteristic polynomial of a square matrix $B$. Let $V_1=\{1,2,\ldots,n+1\}$ and $V_2=\{1^\prime,2^\prime,\ldots,(n+1)^\prime\}$. Then $L(G_n)$ and $\mathcal{L}(G_n)$ can be written as the following two block matrices
\begin{flalign*}
	&\hspace{1cm}{L}(G_n)=\left(
	\begin{array}{cc}
		{L}_{{11}} & {L}_{{12}} \\
		{L}_{{21}} & {L}_{{22}} 
	\end{array}
	\right),
	&
	&\mathcal{L}(G_n)=\left(
	\begin{array}{cc}
		\mathcal{L}_{{11}} & \mathcal{L}_{{12}} \\
		\mathcal{L}_{{21}} & \mathcal{L}_{{22}} 
	\end{array}
	\right),
	&
\end{flalign*}
where ${L}_{{ij}}$ and $\mathcal{L}_{{ij}}$ are the submatrices formed by rows corresponding to vertices in $V_i$ and columns corresponding to vertices in $V_j$ respectively, for $i,j=1,2$. It is easy to check that $L_{11}=L_{22}$, ${L}_{12}=L_{21}$, $\mathcal{L}_{11}=\mathcal{L}_{22}$ and $\mathcal{L}_{12}=\mathcal{L}_{21}$.

Let
\begin{flalign*}
	&\hspace{1cm}T=\left(
	\begin{array}{cc}
		\frac{1}{\sqrt2}I_{n+1} & \frac{1}{\sqrt2}I_{n+1} \\
		\frac{1}{\sqrt2}I_{n+1} & -\frac{1}{\sqrt2}I_{n+1} 
	\end{array}
	\right),
	&
\end{flalign*}
then we have 
\begin{flalign*}
	&\hspace{1cm}T{L}(G_n)T=\left(
	\begin{array}{cc}
		{L}_{A} & 0 \\
		0 & {L}_{S} 
	\end{array}
	\right),
	&
	&T\mathcal{L}(G_n)T=\left(
	\begin{array}{cc}
		\mathcal{L}_{A} & 0 \\
		0 & \mathcal{L}_{S} 
	\end{array}
	\right),
	&
\end{flalign*}
where ${L}_A={L}_{{11}}+{L}_{{12}}$, ${L}_S={L}_{{11}}-{L}_{{12}}$,  $\mathcal{L}_A=\mathcal{L}_{{11}}+\mathcal{L}_{{12}}$ and $\mathcal{L}_S=\mathcal{L}_{{11}}-\mathcal{L}_{{12}}$.

Then similar to the decomposition theorem obtained in \cite{022}, we can get the following decomposition theorem.

\newtheorem{mytheo}{Theorem}[section]
\begin{mytheo}
	\label{lem21}
	Let ${L}_A$, ${L}_S$,  $\mathcal{L}_A$ and $\mathcal{L}_S$ be defined as above. Then
	\begin{flalign*}
		&\hspace{1cm}\Phi(L(G_n))={\Phi({L}_A)}\cdot{\Phi({L}_S)},&  &\Phi(\mathcal{L}(G_n))={\Phi(\mathcal{L}_A)}\cdot{\Phi(\mathcal{L}_S)}.&
	\end{flalign*}
\end{mytheo}

\begin{proof}
	Note that $T^2=I_{2n+2}$, thus $(\det T)^2=1$. Then we have
	\begin{flalign*}
		\hspace{1cm}\Phi(L(G_n))&= \det(xI_{2n+2}-{L}(G_n))
		\nonumber& \\
		&=\det{T}\cdot 
		\det(xI_{2n+2}-{L}(G_n))\cdot{\textrm{det}{T}}\nonumber&\\
		&=\det(xI_{2n+2}-T{L}(G_n)T)\nonumber&\\
		&=\det\left(
		\begin{array}{cc}
			xI_{n+1}-{L}_{A} & \textbf{0} \\
			\textbf{0} & xI_{n+1}-{L}_{S} 
		\end{array}
		\right)\nonumber&\\
		&={\Phi({L}_A)}\cdot{\Phi({L}_S)}.&
	\end{flalign*}
	
	Similarly, we have $\Phi(\mathcal{L}(G_n))={\Phi(\mathcal{L}_A)}\cdot{\Phi(\mathcal{L}_S)}.$
\end{proof}

%Let $0=\mu_1<\mu_2\leqslant\cdots\leqslant\mu_n$ $(n\geqslant2)$ be the Laplacian eigenvalues of $G$ with $n$ vertices.

\begin{mytheo}\label{theo22}
	$(\cite{006}).$
	Let $G$ be a connected graph of order $n$, then $\tau(G)=\frac{1}{n}\prod^n_{i=2}\mu_i$, where $\tau(G)$ is the number of spanning trees of $G$.	
\end{mytheo}

%\begin{mytheo}
%	$([\citealp*{005}]).$
%5	Let $G$ be an $n$-vertex graph with $m$ edges. Then $K^{\prime}(G)=2m\sum^n_{i=2}\frac{1}{\lambda_i}$.	
%\end{mytheo}

\section{Kirchhoff index and multiplicative degree-Kirchhoff index  of $G_n$} 
In this section,  we will determine the Laplacian eigenvalues and the normalized Laplacian eigenvalues of $G_n$ according to  Theorem \ref{lem21}, respectively. Next we will provide a complete description for the sum of the Laplacian (resp. normalized Laplacian) eigenvalues' reciprocals which will be used in calculating the Kirchhoff (resp. multiplicative degree-Kirchhoff) index of $G_n$. Finally, we show that the Kirchhoff (resp. multiplicative degree-Kirchhoff) index of $G_n$ is approximately one quarter of its Wiener (resp. Gutman) index.

\begin{flushleft}
	3.1 Kirchhoff index of $G_n$ 
\end{flushleft}
Note that
\begin{flalign*}
	&\hspace{1cm}{L}_{{11}}=\left(
	\begin{array}{cccccccccc}
		3& -1& 0 & 0& 0& \dots& 0& 0& 0& 0 \\
		-1& 5& -1 & 0& 0& \dots& 0& 0& 0& 0\\
		0& -1 & 5& -1 &  0& \dots& 0& 0& 0& 0\\
		0& 0& -1 &5& -1 &   \dots& 0& 0& 0& 0\\
		0& 0& 0& -1 & 5&    \dots& 0& 0& 0& 0\\
		\vdots& \vdots& \vdots& \vdots& \vdots& \ddots& \vdots& \vdots& \vdots& \vdots\\
		0& 0& 0& 0 & 0&    \dots& 5& -1& 0& 0\\
		0& 0& 0& 0 & 0&    \dots& -1& 5& -1& 0\\
		0& 0& 0& 0 & 0&    \dots& 0& -1& 5& -1\\
		0& 0& 0& 0 & 0&    \dots& 0& 0& -1& 3
	\end{array}
	\right)_{(n+1)\times(n+1)}&
\end{flalign*}
and
\begin{flalign*}
	&\hspace{1cm}{L}_{{12}}=\left(
	\begin{array}{cccccccccc}
		-1& -1& 0 & 0& 0& \dots& 0& 0& 0& 0 \\
		-1& -1& -1 & 0& 0& \dots& 0& 0& 0& 0\\
		0& -1& -1& -1 &  0& \dots& 0& 0& 0& 0\\
		0& 0& -1 & -1& -1 &   \dots& 0& 0& 0& 0\\
		0& 0& 0& -1 & -1&    \dots& 0& 0& 0& 0\\
		\vdots& \vdots& \vdots& \vdots& \vdots& \ddots& \vdots& \vdots& \vdots& \vdots\\
		0& 0& 0& 0 & 0&    \dots& -1& -1& 0& 0\\
		0& 0& 0& 0 & 0&    \dots& -1& -1& -1& 0\\
		0& 0& 0& 0 & 0&    \dots& 0& -1& -1& -1\\
		0& 0& 0& 0 & 0&    \dots& 0& 0& -1& -1
	\end{array}
	\right)_{(n+1)\times(n+1)},&
\end{flalign*}
then
\begin{flalign*}
	&\hspace{1cm}{L}_{A}=\left(
	\begin{array}{cccccccccc}
		2& -2& 0 & 0& 0& \dots& 0& 0& 0& 0 \\
		-2& 4& -2 & 0& 0& \dots& 0& 0& 0& 0\\
		0& -2 & 4& -2 &  0& \dots& 0& 0& 0& 0\\
		0& 0& -2 & 4& -2 &   \dots& 0& 0& 0& 0\\
		0& 0& 0& -2 & 4&    \dots& 0& 0& 0& 0\\
		\vdots& \vdots& \vdots& \vdots& \vdots& \ddots& \vdots& \vdots& \vdots& \vdots\\
		0& 0& 0& 0 & 0&    \dots& 4& -2& 0& 0\\
		0& 0& 0& 0 & 0&    \dots& -2& 4& -2& 0\\
		0& 0& 0& 0 & 0&    \dots& 0& -2& 4& -2\\
		0& 0& 0& 0 & 0&    \dots& 0& 0& -2& 2
	\end{array}
	\right)_{(n+1)\times(n+1)}&
\end{flalign*}
and 
\begin{flalign*}
	&\hspace{1cm}{L}_{S}=\left(
	\begin{array}{cccccccccc}
		4& 0& 0 & 0& 0& \dots& 0& 0& 0& 0 \\
		0&  6 & 0& 0& 0& \dots& 0& 0& 0& 0\\
		0& 0 & 6& 0 &  0& \dots& 0& 0& 0& 0\\
		0& 0& 0 & 6& 0 &   \dots& 0& 0& 0& 0\\
		0& 0& 0& 0 & 6&    \dots& 0& 0& 0& 0\\
		\vdots& \vdots& \vdots& \vdots& \vdots& \ddots& \vdots& \vdots& \vdots& \vdots\\
		0& 0& 0& 0 & 0&    \dots& 6& 0& 0& 0\\
		0& 0& 0& 0 & 0&    \dots& 0& 6& 0& 0\\
		0& 0& 0& 0 & 0&    \dots& 0& 0& 6& 0\\
		0& 0& 0& 0 & 0&    \dots& 0& 0& 0& 4
	\end{array}
	\right)_{(n+1)\times(n+1)}.&
\end{flalign*}

Suppose that the eigenvalues of $L_A$ and $L_S$ are respectively, denoted by $\alpha_i$ ($i=1,2,\dots,n+1$) and $\beta_j$ ($j=1,2,\dots,n+1$) with $\alpha_1\leqslant\alpha_2\leqslant\cdots\leqslant\alpha_{n+1}$ and $\beta_1\leqslant\beta_2\leqslant\cdots\leqslant\beta_{n+1}$. By Theorem \ref{lem21}, the spectrum of $G_n$ is   $\{\alpha_1,\alpha_2,\ldots,\alpha_{n+1},\beta_1,\beta_2,\ldots,\beta_{n+1}\}$. Note that $L_A=2L(P_{n+1})$, then it follows from  \cite{ande} that the eigenvalues of $L_A$ are $\{\alpha_i=8\sin^2[\frac{\pi(i-1)}{2(n+1)}]:i=1,2,\dots,n+1\}$. According to the results obtained in \cite{pan}, we have
\begin{flalign}\label{eq3111}
	&\hspace{1cm}\sum^{n+1}_{i=2}\frac{1}{\alpha_i}=\frac{n(n+2)}{12},&
	&\hspace{1.8cm}\prod^{n+1}_{i=2}{\alpha_i}=(n+1)2^{n}.&
\end{flalign}	
On the other hand, since $L_S$ is a diagonal matrix, thus $\beta_1=\beta_2=4$ and $\beta_3=\beta_4=\cdots=\beta_{n}=\beta_{n+1}=6$. Note that $|V_{G_n}|=2(n+1)$, then we can get the following theorem.

\newtheorem{kkk}{Theorem}[section]
\begin{kkk}
	Let $G_n$ be a linear crossed polyomino chain with $n$ four-order complete graphs. Then
	\begin{flalign*}
		&\hspace{1cm}K\!f(G_n)=\frac{(n+1)(n+2)^2}{6}.&
	\end{flalign*}
\end{kkk}

\begin{proof}
	Note that $|V_{G_n}|=2(n+1)$, then by \eqref{eq11} and \eqref{eq3111}, we have
	\begin{flalign*}
		\hspace{1cm}K\!f(G_n)&=2(n+1)\bigg(\sum^{n+1}_{i=2}\frac{1}{\alpha_i}+\sum^{n+1}_{j=1}\frac{1}{\beta_j}\bigg)\nonumber&\\
		&=2(n+1)\bigg(\frac{n(n+2)}{12}+\frac{n+2}{6}\bigg)\nonumber&\\
		&=\frac{(n+1)(n+2)^2}{6}&
	\end{flalign*}	
	as desired.
\end{proof}

Kirchhoff indices of linear crossed polyomino chains from $G_1$ to $G_{50}$ are listed in Table 1.

Now, we show that the Kirchhoff index of $G_n$ is approximately one quarter of its Wiener index.

\begin{kkk}
	Let $G_n$ be a linear crossed polyomino chain with $n$ four-order complete graphs. Then 
	\begin{flalign*}
		&\hspace{1cm}\lim_{n\to\infty}\frac{K\!f(G_n)}{W(G_n)}=\frac{1}{4}.&
	\end{flalign*}	
\end{kkk}

\begin{proof}
	First we calculate $W(G_n)$. We evaluated $d_{ij}$ for all vertices (fixed $i$ and $j$) (there are two types of vertices) and then added all together and finally divided by two. The expressions of each type of vertices are:\\
	
	\begin{table}[!ht] \centering \caption{{Kirchhoff indices of linear crossed polyomino chains from $G_1$ to $G_{50}$.} }
		\begin{tabular}{llllllllll} \hline
			G& $K\!f(G)$& G& $K\!f(G)$& G& $K\!f(G)$& G& $K\!f(G)$& G& $K\!f(G)$\\ \hline
			$G_1$ & 3.00& $G_{11}$& 338.00& $G_{21}$ & 1939.67& $G_{31}$& 5808.00& $G_{41}$& 12943.00\\
			$G_{2}$ & 8.00& $G_{12}$& 424.67& $G_{22}$ & 2208.00& $G_{32}$& 6358.00& $G_{42}$& 13874.67\\
			$G_{3}$ & 16.67& $G_{13}$& 525.00& $G_{23}$ & 2500.00& $G_{33}$& 6941.67& $G_{43}$& 14850.00\\
			$G_{4}$ & 30.00& $G_{14}$& 640.00& $G_{24}$ & 2816.67& $G_{34}$& 7560.00& $G_{44}$& 15870.00\\
			$G_{5}$ & 49.00& $G_{15}$& 770.67& $G_{25}$ & 3159.00& $G_{35}$& 8214.00& $G_{45}$& 16935.67\\
			$G_{6}$ & 74.67& $G_{16}$& 918.00& $G_{26}$ & 3528.00& $G_{36}$& 8904.67& $G_{46}$& 18048.00\\
			$G_{7}$ & 108.00& $G_{17}$& 1083.00& $G_{27}$ & 3924.67& $G_{37}$& 9633.00& $G_{47}$& 19208.00\\
			$G_{8}$ & 150.00& $G_{18}$& 1266.67& $G_{28}$ & 4350.00& $G_{38}$& 10400.00& $G_{48}$& 20461.67\\
			$G_{9}$ & 201.67& $G_{19}$& 1470.00& $G_{29}$ & 4805.00& $G_{39}$& 11206.67& $G_{49}$& 21675.00\\
			$G_{10}$ & 264.00& $G_{20}$& 1694.00& $G_{30}$ & 5290.67& $G_{40}$& 12054.00& $G_{50}$& 22984.00\\
			\hline
		\end{tabular}
		\label{tab2111}
	\end{table}
	\textbf{\Large$\cdot$} Corner vertex of $G_n$:
	\begin{flalign*}
		&\hspace{1cm}f_1(n)=\sum^{n}_{k=1}{k}\cdot2+1=n^2+n+1.&
	\end{flalign*}
	
	\textbf{\Large$\cdot$} Internal $i$-th ($2\leqslant i\leqslant n$) vertex in $G_n$:
	\begin{flalign*}
		&\hspace{1cm}f_2(i,n)=\sum^{i-1}_{k=1}{k}\cdot2+\sum^{n-i+1}_{k=1}{k}\cdot2+1=n^2-2ni+3n+2i^2-4i+1.&
	\end{flalign*}
	Hence,
	\begin{flalign*}
		&\hspace{1cm}W(G_n)=\frac{4f_1(n)+2\sum^{n}_{i=2}f_2(i,n)}{2}=\frac{2n^3+7n^2+6n+3}{3}.&
	\end{flalign*}
	Together with Theorem 3.1, our result follows immediately.
	%Combining Theorem 3.1 and \eqref{eq311}, one has $\lim_{n\to\infty}\frac{K(G_n)}{W(G_n)}=\frac{1}{4}$.	
\end{proof}

\begin{kkk}
	Let $G_n$ be a linear crossed polyomino chain with $n$ four-order complete graphs. Then 
	\begin{flalign*}
		&\hspace{1cm}\tau(G)=2^{2n+2}\cdot3^{n-1}.&
	\end{flalign*}	
\end{kkk}

\begin{proof}
	Since $|V_{G_{n}}|=2(n+1)$, by  Theorem \ref{theo22} and \eqref{eq3111},  we have	
	\begin{flalign*}
		\hspace{1cm}\tau(G_{n})&=\frac{1}{2(n+1)}\bigg(\prod^{n+1}_{i=2}\alpha_i\cdot\prod^{n+1}_{j=1}\beta_j\bigg)\nonumber&\\
		&=\frac{1}{2(n+1)}\bigg[(n+1)2^n\cdot4^{2}\cdot6^{n-1}\bigg]\nonumber&\\
		&=2^{2n+2}\cdot3^{n-1}.&
	\end{flalign*}	
	This completes  the proof.
\end{proof}

\begin{flushleft}
	3.2 Multiplicative degree-Kirchhoff index of  $G_n$
\end{flushleft}
One can easily obtain  that
\begin{flalign*}
	&\hspace{1cm}\mathcal{L}_{{11}}=\left(
	\begin{array}{cccccccccc}
		1& -\frac{1}{\sqrt{15}}& 0 & 0& 0& \dots& 0& 0& 0& 0 \\
		-\frac{1}{\sqrt{15}}& 1& -\frac{1}{5} & 0& 0& \dots& 0& 0& 0& 0\\
		0& -\frac{1}{5} & 1& -\frac{1}{5} &  0& \dots& 0& 0& 0& 0\\
		0& 0& -\frac{1}{5} & 1& -\frac{1}{5} &   \dots& 0& 0& 0& 0\\
		0& 0& 0& -\frac{1}{5} & 1&    \dots& 0& 0& 0& 0\\
		\vdots& \vdots& \vdots& \vdots& \vdots& \ddots& \vdots& \vdots& \vdots& \vdots\\
		0& 0& 0& 0 & 0&    \dots& 1& -\frac{1}{5}& 0& 0\\
		0& 0& 0& 0 & 0&    \dots& -\frac{1}{5}& 1& -\frac{1}{5}& 0\\
		0& 0& 0& 0 & 0&    \dots& 0& -\frac{1}{5}& 1& -\frac{1}{\sqrt{15}}\\
		0& 0& 0& 0 & 0&    \dots& 0& 0& -\frac{1}{\sqrt{15}}& 1
	\end{array}
	\right)_{(n+1)\times(n+1)}&
\end{flalign*}
and
\begin{flalign*}
	&\hspace{1cm}\mathcal{L}_{{12}}=\left(
	\begin{array}{cccccccccc}
		-\frac{1}{3}& -\frac{1}{\sqrt{15}}& 0 & 0& 0& \dots& 0& 0& 0& 0 \\
		-\frac{1}{\sqrt{15}}& -\frac{1}{5}& -\frac{1}{5} & 0& 0& \dots& 0& 0& 0& 0\\
		0& -\frac{1}{5} & -\frac{1}{5}& -\frac{1}{5} &  0& \dots& 0& 0& 0& 0\\
		0& 0& -\frac{1}{5} & -\frac{1}{5}& -\frac{1}{5} &   \dots& 0& 0& 0& 0\\
		0& 0& 0& -\frac{1}{5} & -\frac{1}{5}&    \dots& 0& 0& 0& 0\\
		\vdots& \vdots& \vdots& \vdots& \vdots& \ddots& \vdots& \vdots& \vdots& \vdots\\
		0& 0& 0& 0 & 0&    \dots& -\frac{1}{5}& -\frac{1}{5}& 0& 0\\
		0& 0& 0& 0 & 0&    \dots& -\frac{1}{5}& -\frac{1}{5}& -\frac{1}{5}& 0\\
		0& 0& 0& 0 & 0&    \dots& 0& -\frac{1}{5}& -\frac{1}{5}& -\frac{1}{\sqrt{15}}\\
		0& 0& 0& 0 & 0&    \dots& 0& 0& -\frac{1}{\sqrt{15}}& -\frac{1}{3}
	\end{array}
	\right)_{(n+1)\times(n+1)}.&
\end{flalign*}
Since $\mathcal{L}_A=\mathcal{L}_{{11}}+\mathcal{L}_{{12}}$ and $\mathcal{L}_S=\mathcal{L}_{{11}}-\mathcal{L}_{{12}}$, then we have
\begin{flalign*}
	&\hspace{1cm}\mathcal{L}_{A}=\left(
	\begin{array}{cccccccccc}
		\frac{2}{3}& -\frac{2}{\sqrt{15}}& 0 & 0& 0& \dots& 0& 0& 0& 0 \\
		-\frac{2}{\sqrt{15}}& \frac{4}{5}& -\frac{2}{5} & 0& 0& \dots& 0& 0& 0& 0\\
		0& -\frac{2}{5} & \frac{4}{5}& -\frac{2}{5} &  0& \dots& 0& 0& 0& 0\\
		0& 0& -\frac{2}{5} & \frac{4}{5}& -\frac{2}{5} &   \dots& 0& 0& 0& 0\\
		0& 0& 0& -\frac{2}{5} & \frac{4}{5}&    \dots& 0& 0& 0& 0\\
		\vdots& \vdots& \vdots& \vdots& \vdots& \ddots& \vdots& \vdots& \vdots& \vdots\\
		0& 0& 0& 0 & 0&    \dots& \frac{4}{5}& -\frac{2}{5}& 0& 0\\
		0& 0& 0& 0 & 0&    \dots& -\frac{2}{5}& \frac{4}{5}& -\frac{2}{5}& 0\\
		0& 0& 0& 0 & 0&    \dots& 0& -\frac{2}{5}& \frac{4}{5}& -\frac{2}{\sqrt{15}}\\
		0& 0& 0& 0 & 0&    \dots& 0& 0& -\frac{2}{\sqrt{15}}& \frac{2}{3}
	\end{array}
	\right)_{(n+1)\times(n+1)}&
\end{flalign*}
and 
\begin{flalign*}
	&\hspace{1cm}\mathcal{L}_{S}=\left(
	\begin{array}{cccccccccc}
		\frac{4}{3}& 0& 0 & 0& 0& \dots& 0& 0& 0& 0 \\
		0&  \frac{6}{5} & 0& 0& 0& \dots& 0& 0& 0& 0\\
		0& 0 & \frac{6}{5}& 0 &  0& \dots& 0& 0& 0& 0\\
		0& 0& 0 & \frac{6}{5}& 0 &   \dots& 0& 0& 0& 0\\
		0& 0& 0& 0 & \frac{6}{5}&    \dots& 0& 0& 0& 0\\
		\vdots& \vdots& \vdots& \vdots& \vdots& \ddots& \vdots& \vdots& \vdots& \vdots\\
		0& 0& 0& 0 & 0&    \dots& \frac{6}{5}& 0& 0& 0\\
		0& 0& 0& 0 & 0&    \dots& 0& \frac{6}{5}& 0& 0\\
		0& 0& 0& 0 & 0&    \dots& 0& 0& \frac{6}{5}& 0\\
		0& 0& 0& 0 & 0&    \dots& 0& 0& 0& \frac{4}{3}
	\end{array}
	\right)_{(n+1)\times(n+1)}.&
\end{flalign*}

Suppose that the eigenvalues of $\mathcal{L}_A$ and $\mathcal{L}_S$ are respectively, denoted by $\gamma_i$ ($i=1,2,\dots,n+1$) and $\eta_j$ ($j=1,2,\dots,n+1$) with $\gamma_1\leqslant\gamma_2\leqslant\cdots\leqslant\gamma_{n+1}$ and $\eta_1\leqslant\eta_2\leqslant\cdots\leqslant\eta_{n+1}$.  Then the eigenvalues of $\mathcal{L}{(G_n)}$ is $\{\gamma_1,\gamma_2,\ldots,\gamma_{n+1},\eta_1,\eta_2,\ldots,\eta_{n+1}\}$. Note that $\mathcal{L}_S$ is a diagonal matrix, then  $\eta_1=\eta_2=\cdots=\eta_{n-2}=\eta_{n-1}=\frac{6}{5}$ and $\eta_{n}=\eta_{n+1}=\frac{4}{3}$. For $\xi=(\sqrt3,\sqrt5,\ldots,\sqrt5,\sqrt3)^T$,  $\mathcal{L}_A\xi=0$, then 0 is an eigenvalue of $\mathcal{L}_A$, which implies that $\gamma_1=0$ and $\gamma_2>0$. Note that $|E_{G_n}|=5n+1$, then by \eqref{eq111}, we have
\begin{flalign}\label{eq31}
	&\hspace{1cm}K\!f^{*}(G_n)=2(5n+1)\bigg(\sum^{n+1}_{i=2}\frac{1}{\gamma_i}+\sum^{n+1}_{j=1}\frac{1}{\eta_j}\bigg)=2(5n+1)\bigg(\sum^{n+1}_{i=2}\frac{1}{\gamma_i}+\frac{5n+4}{6}\bigg).&
\end{flalign}

Based on the relationship between the roots and the coefficients of $\Phi(\mathcal{L}_A)$, the formula of $\sum^{n+1}_{i=2}\frac{1}{\gamma_i}$ is derived in the next  theorem.

\begin{kkk}\label{the31}
	Let $\gamma_i$ ($i=1,2,\ldots,n+1$) be defined as above. Then
	\begin{flalign}\label{eq30}
		&\hspace{1cm}\sum^{n+1}_{i=2}\frac{1}{\gamma_i}=\frac{n(25n^2+15n+14)}{12(5n+1)}.&
	\end{flalign}	
\end{kkk}

\begin{proof}
	Suppose that $\Phi(\mathcal{L}_A)=x^{n+1}+a_1x^n+\cdots+a_{n-1}x^2+a_nx=x(x^n+a_1x^{n-1}+\cdots+a_{n-1}x+a_n)$. Then $\gamma_i$ ($i=2,3,\ldots,n+1$) satisfies the following equation
	\begin{flalign*}
		&\hspace{1cm}x^n+a_1x^{n-1}+\cdots+a_{n-1}x+a_n=0.&
	\end{flalign*}	
	So $\frac{1}{\gamma_i}$ ($i=2,3\ldots,n+1$) satisfies the following equation
	\begin{flalign*}
		&\hspace{1cm}a_nx^n+a_{n-1}x^{n-1}+\cdots+a_{1}x+1=0.&
	\end{flalign*}
	By Vieta's Theorem, we have
	\begin{flalign}\label{eq32}
		&\hspace{1cm}\sum^{n+1}_{i=2}\frac{1}{\gamma_i}=\frac{(-1)^{n-1}a_{n-1}}{(-1)^na_n}.&
	\end{flalign}

	For $1\leqslant i\leqslant n$, let $C_i$ be the $i$-th order principal submatrix formed by the first $i$ rows and columns of $\mathcal{L}_A$ and $c_i=\det{C_i}$. Obviously, $c_1=\frac{2}{3}$ and $c_2=\frac{4}{15}$. For $3\leqslant i\leqslant n$, expanding $\det{C_i}$ with regard to its last row yields that $c_i=\frac{4}{5}c_{i-1}-\frac{4}{25}c_{i-2}$.
	Then, we have 
	\begin{flalign}\label{eq33}
		&\hspace{1cm}c_i=\frac{5}{3}\cdot\bigg({\frac{2}{5}}\bigg)^i.&
	\end{flalign}	
	
	Next, we will determine the expressions of $(-1)^na_n$ and $(-1)^{n-1}a_{n-1}$, respectively. For convenience, let the diagonal entries of $\mathcal{L}_A$ be $k_{ii}$ and $c_0$ be $1$.\\
	
	\textbf{Claim 1.}  $(-1)^na_n=\frac{25n+5}{9}\cdot\big(\frac{2}{5}\big)^n$.\\
	
	Since the number $(-1)^na_n$ is the sum of those principal minors of $\mathcal{L}_A$ which have $n$ rows and columns, we have
	\begin{flalign}\label{eq34}
		&\hspace{1cm}(-1)^na_n=\sum^{n+1}_{i=1}\left|
		\begin{array}{cccccccc}
			k_{11}& -\frac{2}{\sqrt{15}}&  \dots& 0& 0& 0& 0& 0 \\
			-\frac{2}{\sqrt{15}}& k_{22}&  \dots& 0& 0& 0& 0& 0\\
			\vdots& \vdots& \ddots& \vdots& \vdots& \vdots& \vdots& \vdots\\
			0& 0& \dots& k_{i-1,i-1} & 0& 0 &   \dots& 0\\
			0& 0& \dots& 0 & k_{i+1,i+1}&  -\frac{2}{5} &   \dots& 0\\
			\vdots& \vdots& \vdots& \vdots& \vdots& \ddots& \vdots& \vdots\\
			0& 0& \dots& 0 & 0&    \dots& k_{n,n}&  -\frac{2}{\sqrt{15}}\\
			0& 0& \dots& 0 & 0&   \dots& -\frac{2}{\sqrt{15}}&  k_{n+1,n+1}
		\end{array}
		\right|
		\nonumber&\\
		&\hspace{2.1cm}=\sum^{n+1}_{i=1}\left({\left|
			\begin{array}{cccc}
				k_{11}& -\frac{2}{\sqrt{15}}&  \dots& 0 \\
				-\frac{2}{\sqrt{15}}& k_{22}&  \dots& 0\\
				\vdots& \vdots& \ddots& \vdots\\
				0& 0& \cdots&  k_{i-1,i-1}
			\end{array}
			\right|}{\left|
			\begin{array}{cccc}
				k_{i+1,i+1}& -\frac{2}{5}&  \dots& 0 \\
				\vdots& \ddots& \vdots& \vdots\\
				0& \dots&  k_{n,n}&  -\frac{2}{\sqrt{15}}\\
				0&  \dots& -\frac{2}{\sqrt{15}}&  k_{n+1,n+1}
			\end{array}
			\right|}
		\right).&
	\end{flalign}
	Note that a permutation similarity transformation of a square matrix preserves its determinant. Together with the property of $\mathcal{L}_A$, the right hand side of \eqref{eq34} is equal to $\det{C_{n+1-i}}$. By \eqref{eq33}, we have
	\begin{flalign*}
		\hspace{1cm}(-1)^na_n&=\sum^{n+1}_{i=1}c_{i-1}c_{n+1-i}\nonumber&\\
		&=2c_n+\sum^{n}_{i=2}c_{i-1}c_{n+1-i}\nonumber&\\
		&=2\cdot\frac{5}{3}\bigg(\frac{2}{5}\bigg)^n+\sum^{n}_{i=2}\frac{5}{3}\bigg(\frac{2}{5}\bigg)^{i-1}\cdot\frac{5}{3}\bigg(\frac{2}{5}\bigg)^{n+1-i}\nonumber&\\
		&=\frac{25n+5}{9}\cdot\bigg(\frac{2}{5}\bigg)^n.&
	\end{flalign*}

	\textbf{Claim 2.}  $(-1)^{n-1}a_{n-1}=\frac{n(25n^2+15n+14)}{54}\cdot\big(\frac{2}{5}\big)^{n-1}$.\\
	
	Note that the number $(-1)^{n-1}a_{n-1}$ is the sum of those principal minors of $\mathcal{L}_A$ which have $n-1$ rows and columns, hence $(-1)^{n-1}a_{n-1}$ equals 
	
	\begin{flalign*}
		&\sum_{1\leqslant i<j\leqslant n+1}\left|
		\begin{array}{cccccccccccc}
			k_{11}& -\frac{2}{\sqrt{15}}&  \dots& 0& 0& 0& \dots& 0& 0& \dots& 0& 0 \\
			-\frac{2}{\sqrt{15}}& k_{22}&  \dots& 0& 0& 0& \dots& 0& 0& \dots& 0& 0 \\
			\vdots& \vdots& \ddots& \vdots& \vdots& \vdots& \dots& \vdots& \vdots& \dots& \vdots& \vdots\\
			0& 0&  \dots& k_{i-1,i-1}& 0& 0& \dots& 0& 0& \dots& 0& 0 \\
			0& 0&  \dots& 0& k_{i+1,i+1}& -\frac{2}{5}& \dots& 0& 0& \dots& 0& 0 \\		
			0& 0&  \dots& 0& -\frac{2}{5}& k_{i+2,i+2}& \dots& 0& 0& \dots& 0& 0 \\			
			\vdots& \vdots& \dots& \vdots& \vdots& \vdots& \ddots& \vdots& \vdots& \dots& \vdots& \vdots\\	
			0& 0&  \dots& 0& 0& 0& \dots& k_{j-1,j-1}& 0& \dots& 0& 0 \\		
			0& 0&  \dots& 0& 0& 0& \dots& 0& k_{j+1,j+1}&  \dots& 0& 0 \\	
			\vdots& \vdots& \dots& \vdots& \vdots& \vdots& \dots& \vdots& \vdots& \ddots& \vdots& \vdots\\		
			0& 0&  \dots& 0& 0& 0& \dots& 0& 0&  \dots& k_{n,n}& -\frac{2}{\sqrt{15}} \\	
			0& 0&  \dots& 0& 0& 0& \dots& 0& 0&  \dots& -\frac{2}{\sqrt{15}}& k_{n+1,n+1} 
		\end{array}
		\right|
		\nonumber&\\
		&=\sum_{1\leqslant i<j\leqslant n+1}{\left|
			\begin{array}{cccc}
				k_{11}& -\frac{2}{\sqrt{15}}&  \dots& 0 \\
				-\frac{2}{\sqrt{15}}& k_{22}&  \dots& 0\\
				\vdots& \vdots& \ddots& \vdots\\
				0& 0& \dots&  k_{i-1,i-1}
			\end{array}
			\right|}\cdot{\left|
			\begin{array}{cccc}
				k_{i+1,i+1}& -\frac{2}{5}&  \dots& 0 \\
				-\frac{2}{5}& k_{i+2,i+2}&  \dots&  0\\
				\vdots& \vdots& \ddots& \vdots\\
				0&  0& \dots& k_{j-1,j-1}
			\end{array}
			\right|}\nonumber&\\
		&~~~~~\cdot{\left|
			\begin{array}{cccc}
				k_{j+1,j+1}&   \dots& 0& 0 \\
				\vdots& \ddots& \vdots& \vdots\\
				0&   \dots& k_{n,n}& -\frac{2}{\sqrt{15}}\\
				0&  \dots& -\frac{2}{\sqrt{15}}&  k_{n+1,n+1}
			\end{array}
			\right|}&
	\end{flalign*}
	Similarly, the right hand side of the above equation is equal to $\det{C_{n+1-j}}$. Then, we have
	\begin{flalign}\label{eq35}
		&\hspace{1cm}(-1)^{n-1}a_{n-1}=\sum_{1\leqslant i<j\leqslant n+1}c_{i-1}c_{n+1-j}\cdot\det{M}_{ij},&
	\end{flalign}	
	where
	\begin{flalign*}
		&\hspace{1cm} {M}_{ij}=\left(
		\begin{array}{rrrrrr}
			\frac{4}{5}& -\frac{2}{5}& 0 & \dots& 0& 0 \\
			-\frac{2}{5}& \frac{4}{5}& -\frac{2}{5}&  \dots& 0& 0 \\
			0 & -\frac{2}{5}& \frac{4}{5}&  \dots& 0& 0 \\
			\vdots& \vdots& \vdots& \ddots& \vdots& \vdots\\
			0& 0& 0&  \dots&  \frac{4}{5}&  -\frac{2}{5}\\
			0& 0& 0&  \dots&   -\frac{2}{5}& \frac{4}{5}
		\end{array}
		\right)_{(j-i-1)\times(j-i-1)}&
	\end{flalign*}
	It is easy to check that $\det{M}_{ij}=\big(\frac{2}{5}\big)^{j-i-1}(j-i)$. Thus in view of \eqref{eq35}, we have
	\begin{flalign*}
	(-1)^{n-1}a_{n-1}&=\sum_{1\leqslant i<j\leqslant n+1}\bigg(\frac{2}{5}\bigg)^{j-i-1}(j-i)\cdot{c_{i-1}}{c_{n+1-j}}
		\nonumber& \\
		&=\sum^{n}_{i=2}\bigg(\frac{2}{5}\bigg)^{n-i}(n+1-i)\cdot{c_{i-1}}+\sum^{n}_{j=2}\bigg(\frac{2}{5}\bigg)^{j-2}(j-1)\cdot{c_{n+1-j}}+\sum_{2\leqslant i<j\leqslant n}\bigg(\frac{2}{5}\bigg)^{j-i-1}(j-i)\cdot{c_{i-1}}{c_{n+1-j}}+n\bigg(\frac{2}{5}\bigg)^{n-1}\nonumber&\\
		&= 2\sum^{n}_{i=2}\bigg(\frac{2}{5}\bigg)^{n-i}(n+1-i)\cdot{c_{i-1}}+\sum_{2\leqslant i<j\leqslant n}\bigg(\frac{2}{5}\bigg)^{j-i-1}(j-i)\cdot{c_{i-1}}{c_{n+1-j}}+n\bigg(\frac{2}{5}\bigg)^{n-1}&\\
	\end{flalign*}
	
	By \eqref{eq33}, we have
	\begin{flalign*}
		&\hspace{1cm}\sum^n_{i=2}\bigg(\frac{2}{5}\bigg)^{n-i}(n+1-i)\cdot{c_{i-1}}=\sum^{n}_{i=2}\bigg(\frac{2}{5}\bigg)^{n-i}(n+1-i)\cdot\frac{5}{3}\bigg(\frac{2}{5}\bigg)^{i-1}=\frac{5n(n-1)}{6}\bigg(\frac{2}{5}\bigg)^{n-1}&
	\end{flalign*}	
	and
	\begin{flalign*}
		\hspace{1cm}\sum_{2\leqslant i<j\leqslant n}\bigg(\frac{2}{5}\bigg)^{j-i-1}(j-i)\cdot{c_{i-1}}{c_{n+1-j}}&=\sum_{2\leqslant i<j\leqslant n}\bigg(\frac{2}{5}\bigg)^{j-i-1}\cdot\bigg(\frac{2}{5}\bigg)^{i-1}\cdot\bigg(\frac{2}{5}\bigg)^{n+1-j}\frac{25(j-i)}{9}\nonumber&\\
		&=\frac{25}{9}\bigg(\frac{2}{5}\bigg)^{n-1}\sum_{2\leqslant i<j\leqslant n}(j-i) \nonumber&\\
		&=\frac{25n(n-1)(n-2)}{54}\bigg(\frac{2}{5}\bigg)^{n-1}.&
	\end{flalign*}

	Hence,
	\begin{flalign*}
		&(-1)^{n-1}a_{n-1}=\frac{5n(n-1)}{3}\bigg(\frac{2}{5}\bigg)^{n-1}+\frac{25n(n-1)(n-2)}{54}\bigg(\frac{2}{5}\bigg)^{n-1}+n\bigg(\frac{2}{5}\bigg)^{n-1}=\frac{n(25n^2+15n+14)}{54}\bigg(\frac{2}{5}\bigg)^{n-1}.&
	\end{flalign*}	
	
	Substituting Claims 1-2 into \eqref{eq32} yields that $\sum^{n+1}_{i=2}\frac{1}{\gamma_i}=\frac{n(25n^2+15n+14)}{12(5n+1)}$.	
\end{proof}

Together with \eqref{eq31} and \eqref{eq30}, we can get the following theorem immediately.

\begin{kkk}
	Let $G_n$ be a linear crossed polyomino chain with $n$ four-order complete graphs. Then 
	\begin{flalign}\label{eq388}
		&\hspace{1cm}K\!f^{*}(G_n)=\frac{25n^3+65n^2+64n+8}{6}.&
	\end{flalign}	
\end{kkk}

Multiplicative degree-Kirchhoff indices of linear crossed polyomino chains from $G_1$ to $G_{50}$ are listed in Table 2. 

\begin{table}[!ht] \centering \caption{{Multiplicative  degree-Kirchhoff indices of linear crossed polyomino chains from $G_1$ to $G_{50}$.} }
	\begin{tabular}{llllllllll} \hline
		G& $K\!f^{*}(G)$& G& $K\!f^{*}(G)$& G& $K\!f^{*}(G)$& G& $K\!f^{*}(G)$& G& $K\!f^{*}(G)$\\ \hline
		$G_1$ & 27.00& $G_{11}$& 6975.33& $G_{21}$ & 43590.33& $G_{31}$& 134872.00& $G_{41}$& 305820.33\\
		$G_{2}$ & 99.33& $G_{12}$& 8889.33& $G_{22}$ & 49846.00& $G_{32}$& 147969.33& $G_{42}$& 328259.33\\
		$G_{3}$ & 243.33& $G_{13}$& 11125.00& $G_{23}$ & 56673.33& $G_{33}$& 161888.33& $G_{43}$& 351770.00\\
		$G_{4}$ & 484.00& $G_{14}$& 13707.33& $G_{24}$ & 64097.33& $G_{34}$& 176654.00& $G_{44}$& 376377.33\\
		$G_{5}$ & 846.33& $G_{15}$& 16661.33& $G_{25}$ & 72143.00& $G_{35}$& 192291.33& $G_{45}$& 402106.33\\
		$G_{6}$ & 1355.33& $G_{16}$& 20012.00& $G_{26}$ & 80835.33& $G_{36}$& 208825.33& $G_{46}$& 428982.00\\
		$G_{7}$ & 2036.00& $G_{17}$& 23784.33& $G_{27}$ & 90199.33& $G_{37}$& 226281.00& $G_{47}$& 457029.33\\
		$G_{8}$ & 2913.33& $G_{18}$& 28003.33& $G_{28}$ & 100260.00& $G_{38}$& 244683.33& $G_{48}$& 486273.33\\
		$G_{9}$ & 4012.33& $G_{19}$& 32694.00& $G_{29}$ & 111042.33& $G_{39}$& 264057.33& $G_{49}$& 516739.00\\
		$G_{10}$ & 5358.00& $G_{20}$& 37881.33& $G_{30}$ & 122571.33& $G_{40}$& 184428.00& $G_{50}$& 548451.33\\
		\hline
	\end{tabular}
	\label{tab21}
\end{table}

Finally, we show that the multiplicative degree-Kirchhoff index of $G_n$ is approximately one quarter of its Gutman index.

\begin{kkk}
	Let $G_n$ be a linear crossed polyomino chain with $n$ four-order complete graphs. Then 
	\begin{flalign*}
		&\hspace{1cm}\lim_{n\to\infty}\frac{K\!f^{*}(G_n)}{\textrm{Gut}(G_n)}=\frac{1}{4}.&
	\end{flalign*}	
\end{kkk}

\begin{proof}
	First we calculate $\textrm{Gut}(G_n)$. We evaluated $d_id_jd_{ij}$ for all vertices (fixed $i$ and $j$) (there are two types of vertices) and then added all together and finally divided by two. The expressions of each type of vertices are:\\
	
	\textbf{\Large$\cdot$} Corner vertex of $G_n$:
	\begin{flalign*}
		&\hspace{1cm}g_1(n)=\sum^{n-1}_{k=1}3\cdot5\cdot{k}\cdot2+3\cdot3\cdot(1+n+n)=15n^2+3n+9.&
	\end{flalign*}
	
	\textbf{\Large$\cdot$} Internal $i$-th ($2\leqslant i\leqslant n$) vertex in $G_n$:
	\begin{flalign*}
		&\hspace{1cm}g_2(i,n)=\sum^{i-2}_{k=1}5\cdot5\cdot{k}\cdot2+\sum^{n-i}_{k=1}5\cdot5\cdot{k}\cdot2+5\cdot5\cdot1+5\cdot3\cdot(i-1+n+1-i)\cdot2\nonumber&\\
		&\hspace{2cm}=25n^2-50ni+55n+50i^2-100i+75.&
	\end{flalign*}
	Hence,
	\begin{flalign*}
		&\hspace{1cm}\textrm{Gut}(G_n)=\frac{4g_1(n)+2\sum^{n}_{i=2}g_2(i,n)}{2}=\frac{50n^3+30n^2+103n-21}{3}.&
	\end{flalign*}
	Together with Theorem \ref{the31}, our result follows immediately.
\end{proof}

\section{Kirchhoff index and number of spanning trees of $G^r_n$}
In this section, for any graph $G^r_n\in \mathcal{G}^r_{n}$, we will determine its Kirchhoff index and number of spanning trees. Furthermore, we will show that the Kirchhoff index of $G^r_n$ is approximately one quarter of its Wiener index. 

Similar to the Laplacian polynomial decomposition theorem of $G_n$, we can obtain that the Laplacian spectrum of $G^r_n$ consists of the eigenvalues of both $L_A(G^r_n)$ and $L_S(G^r_n)$. It is routine to check that
\begin{flalign*}
	&\hspace{1cm}{L}_{{11}}(G^r_n)=\left(
	\begin{array}{ccccccc}
		l_{11}& -1& 0 & \dots&  0& 0& 0 \\
		-1& l_{22}& -1 &  \dots&  0& 0& 0\\
		0& -1 & l_{33}& \dots&  0& 0& 0\\
		\vdots& \vdots& \vdots&  \ddots& \vdots&  \vdots& \vdots\\
		0& 0& 0&     \dots& l_{n-1,n-1}& -1& 0\\
		0& 0& 0&     \dots&  -1& l_{n,n}& -1\\
		0& 0& 0&     \dots& 0& -1& l_{n+1,n+1}
	\end{array}
	\right)_{(n+1)\times(n+1)}&
\end{flalign*}
and
\begin{flalign*}
	&\hspace{1cm}{L}_{{12}}(G^r_n)=\left(
	\begin{array}{ccccccc}
		t_{11}& -1& 0 &  \dots& 0& 0& 0 \\
		-1& t_{22}& -1 &  \dots&  0& 0& 0\\
		0& -1& t_{33}&  \dots&  0& 0& 0\\
		\vdots& \vdots& \vdots&  \ddots&  \vdots& \vdots& \vdots\\
		0& 0& 0&     \dots&  t_{n-1,n-1}& -1& 0\\
		0& 0& 0&    \dots&  -1& t_{n,n}& -1\\
		0& 0& 0&     \dots& 0& -1& t_{n+1,n+1}
	\end{array}
	\right)_{(n+1)\times(n+1)},&
\end{flalign*}
where $l_{ii}=d_i$, $t_{ii}=0$ if $d_{i}\in\{2,4\}$ and $t_{ii}=-1$ if $d_{i}\in\{3,5\}$.

Note that $l_{11}+t_{11}=l_{n+1,n+1}+t_{n+1,n+1}=2$ and $l_{ii}+t_{ii}=4$ for $2\leqslant i\leqslant n$. Then we have
\begin{flalign*}
	&\hspace{1cm}{L}_{A}(G^r_n)=\left(
	\begin{array}{ccccccc}
		2& -2& 0 &  \dots&  0& 0& 0 \\
		-2& 4& -2 &  \dots&  0& 0& 0\\
		0& -2 & 4&  \dots&  0& 0& 0\\
		\vdots& \vdots&  \vdots& \ddots&  \vdots& \vdots& \vdots\\
		0& 0& 0&    \dots&  4& -2& 0\\
		0& 0& 0&     \dots&  -2& 4& -2\\
		0& 0& 0&     \dots&  0& -2& 2
	\end{array}
	\right)_{(n+1)\times(n+1)}.&
\end{flalign*}
In addition, since $l_{jj}-t_{jj}=2\big\lfloor\frac{d_{j}+1}{2}\big\rfloor$ $(1\leqslant j\leqslant n+1)$, then
\begin{flalign*}
	&\hspace{1cm}{L}_{S}(G^r_n)=\left(
	\begin{array}{ccccccc}
		s_{11}& 0& 0 &  \dots&  0& 0& 0 \\
		0&  s_{22} & 0&  \dots&  0& 0& 0\\
		0& 0 & s_{33}&  \dots&  0& 0& 0\\
		\vdots& \vdots& \vdots&  \ddots& \vdots&  \vdots& \vdots\\
		0& 0& 0&     \dots&  s_{n-1,n-1}& 0& 0\\
		0& 0& 0&    \dots&  0& s_{n,n}& 0\\
		0& 0& 0&   \dots&  0& 0& s_{n+1,n+1}
	\end{array}
	\right)_{(n+1)\times(n+1)},&
\end{flalign*}
where $s_{jj}=2\big\lfloor\frac{d_{j}+1}{2}\big\rfloor$, $j=1,2,\ldots,n+1$.

Let $\zeta_j$ ($j=1,2,\ldots,n+1$) be the eigenvalues of  $L_S(G^r_n)$ with $\zeta_1\leqslant\zeta_2\leqslant\cdots\leqslant\zeta_{n+1}$. Then one can easily obtain that the spectrum of $G^r_n$ is   $\{\alpha_1,\alpha_2,\ldots,\alpha_{n+1},\zeta_1,\zeta_2,\ldots,\zeta_{n+1}\}$. The expressions of $\sum^{n+1}_{j=1}\frac{1}{\zeta_j}$  and $\prod^{n+1}_{j=1}\zeta_j$ are derived in the next theorem.

\begin{kkk}\label{theo41}
	For any graph $G^r_{n}\in\mathcal{G}^r_{n}$, we have
	\begin{flalign*}
		&\hspace{1cm}\sum^{n+1}_{j=1}\frac{1}{\zeta_j}=\frac{2n+r-d+16}{12},&
		&\hspace{1.8cm}\prod^{n+1}_{j=1}{\zeta_j}=2^{n+r+2d-9}\cdot3^{n-r-d+5}.&
	\end{flalign*}	
	where	$d=d_{1}+d_{{n+1}}$. 
\end{kkk}

\begin{proof}
	It is easy to observe that the eigenvalues of $L_S(G^r_n)$ are $s_{11},s_{22},\ldots,s_{n+1,n+1}$, where $s_{jj}=2\big\lfloor\frac{d_{j}+1}{2}\big\rfloor$, $j=1,2,\ldots,n+1$. \\
	
	$\mathbf{Case~1}~~~$ If $d_{1}=d_{{n+1}}=2$, then the edges $11^{\prime}$ and ${(n+1)}{(n+1)^{\prime}}$ are deleted, which implies that $r\geqslant2$. In this case,  we can get that  $\zeta_1=\zeta_2=2$, $\zeta_3=\zeta_4=\cdots=\zeta_{r}=4$ and $\zeta_{r+1}=\zeta_{r+2}=\cdots=\zeta_{n+1}=6$. Hence, we have
	\begin{flalign*}
		&\hspace{1cm}\sum^{n+1}_{j=1}\frac{1}{\zeta_j}=\frac{2}{2}+\frac{r-2}{4}+\frac{n+1-r}{6}=\frac{2n+r+8}{12},&
		&\hspace{1cm}\prod^{n+1}_{j=1}\zeta_j=2^2\cdot4^{r-2}\cdot 6^{n+1-r}=2^{n+r-1}\cdot3^{n+1-r}.&
	\end{flalign*}
	
	$\mathbf{Case~2}~~~$ If $d_{1}=d_{{n+1}}=3$, then the edges $11^{\prime}$ and ${(n+1)}{(n+1)^{\prime}}$  are not deleted,  which implies that $r\geqslant0$. Therefore, we can obtain that $\zeta_1=\zeta_2=\cdots=\zeta_{r+2}=4$ and $\zeta_{r+3}=\zeta_{r+4}=\cdots=\zeta_{n+1}=6$. Thus, we have
	\begin{flalign*}
		&\hspace{1cm}\sum^{n+1}_{j=1}\frac{1}{\zeta_j}=\frac{r+2}{4}+\frac{n-r-1}{6}=\frac{2n+r+4}{12},&
		&\hspace{1cm}\prod^{n+1}_{j=1}\zeta_j=4^{r+2}\cdot 6^{n-1-r}=2^{n+r+3}\cdot3^{n-1-r}.&
	\end{flalign*}

	$\mathbf{Case~3}~~~$ If $d_{1}=2$ and $d_{{n+1}}=3$, or  $d_{1}=3$ and $d_{{n+1}}=2$,  that is the edge  $11^{\prime}$ or ${(n+1)}{(n+1)^{\prime}}$  is deleted, thus we have $r\geqslant1$. One can immediately get that $\zeta_1=2$,  $\zeta_2=\zeta_3=\cdots=\zeta_{r+1}=4$ and $\zeta_{r+2}=\zeta_{r+3}=\cdots=\zeta_{n+1}=6$. In this case, we can obtain that
	\begin{flalign*}
		&\hspace{1cm}\sum^{n+1}_{j=1}\frac{1}{\zeta_j}=\frac{1}{2}+\frac{r}{4}+\frac{n-r}{6}=\frac{2n+r+6}{12},&
		&\hspace{1cm}\prod^{n+1}_{j=1}\zeta_j=2\cdot4^{r}\cdot 6^{n-r}=2^{n+r+1}\cdot3^{n-r}.&
	\end{flalign*}
	
	Above all, we have
	\begin{flalign*}
		&\hspace{1cm}\sum^{n+1}_{j=1}\frac{1}{\zeta_j}=\frac{2n+r-2d+16}{12},&
		&\hspace{1.6cm}\prod^{n+1}_{j=1}{\zeta_j}=2^{n+r+2d-9}\cdot3^{n-r-d+5}.&
	\end{flalign*}
	
	This completes the proof.
\end{proof} 

\begin{kkk}
	For any graph $G^r_{n}\in\mathcal{G}^r_{n}$, we have
	\begin{flalign}\label{eq411}
		&\hspace{1cm}K\!f(G^r_{n})=\frac{(n+1)(n^2+4n+r-2d+16)}{6},&
	\end{flalign}
	where	$d=d_{1}+d_{{n+1}}$.
\end{kkk}

\begin{proof}
	Since $|V_{G^r_{n}}|=2(n+1)$, by  \eqref{eq11}, \eqref{eq3111} and Theorem \ref{theo41},  we have
	\begin{flalign*}
		\hspace{1cm}K\!f(G^r_{n})&=2(n+1)\bigg(\sum^{n+1}_{i=2}\frac{1}{\alpha_i}+\sum^{n+1}_{j=1}\frac{1}{\zeta_j}\bigg)\nonumber&\\
		&=2(n+1)\bigg[\frac{n(n+2)}{12}+\frac{2n+r-2d+16}{12}\bigg]\nonumber&\\
		&=\frac{(n+1)(n^2+4n+r-2d+16)}{6}.&
	\end{flalign*}	
	This completes the proof.
\end{proof}

\begin{kkk}
	For any graph $G^r_{n}\in\mathcal{G}^r_{n}$, we have
	\begin{flalign*}
		\label{4.3}
		&\hspace{1cm}\lim_{n\to\infty}\frac{K\!f(G^r_{n})}{W(G^r_{n})}=\frac{1}{4}.&
	\end{flalign*}
\end{kkk}

\begin{proof}
	For any graph $G^r_{n}\in\mathcal{G}^r_{n}$, one can easily check  that $W(G^r_{n})=W(G_n)+r$. By \eqref{eq388}, we have
	\begin{flalign*}
		&\hspace{1cm}W(G^r_{n})=\frac{2n^3+7n^2+6n+3r+3}{3}.&
	\end{flalign*}		
	Together with \eqref{eq411}, our result follows immediately.
\end{proof}

\begin{kkk}
	For any graph $G^r_{n}\in\mathcal{G}^r_{n}$, we have
	\begin{flalign*}
		&\hspace{1cm}\tau(G^r_{n})=2^{2n+r+2d-10}\cdot3^{n-r-d+5},&
	\end{flalign*}
	where	$d=d_{1}+d_{{n+1}}$.
\end{kkk}

\begin{proof}
	Since $|V_{G^r_{n}}|=2(n+1)$, by  Theorem \ref{theo22}, \eqref{eq3111} and Theorem \ref{theo41},  we have\\
	
	$\mathbf{Case~1}~~~$ If $d_{1}=d_{{n+1}}=2$, then	
	\begin{flalign*}
		\hspace{1cm}\tau(G^r_{n})&=\frac{1}{2(n+1)}\bigg(\prod^{n+1}_{i=2}\alpha_i\cdot\prod^{n+1}_{j=1}\zeta_j\bigg)\nonumber&\\
		&=\frac{1}{2(n+1)}\bigg[(n+1)2^n\cdot2^{n+r-1}\cdot3^{n+1-r}\bigg]\nonumber&\\
		&=2^{2n+r-2}\cdot3^{n+1-r}.&
	\end{flalign*}	
	
	$\mathbf{Case~2}~~~$ If $d_{1}=d_{{n+1}}=3$, then 	
	\begin{flalign*}
		\hspace{1cm}\tau(G^r_{n})&=\frac{1}{2(n+1)}\bigg(\prod^{n+1}_{i=2}\alpha_i\cdot\prod^{n+1}_{j=1}\zeta_j\bigg)\nonumber&\\
		&=\frac{1}{2(n+1)}\bigg[(n+1)2^n\cdot2^{n+r+3}\cdot3^{n-1-r}\bigg]\nonumber&\\
		&=2^{2n+r+2}\cdot3^{n-1-r}.&
	\end{flalign*}
	
	$\mathbf{Case~3}~~~$ If $d_{1}=2$ and $d_{{n+1}}=3$, or  $d_{1}=3$ and $d_{{n+1}}=2$,  then
	\begin{flalign*}
		\hspace{1cm}\tau(G^r_{n})&=\frac{1}{2(n+1)}\bigg(\prod^{n+1}_{i=2}\alpha_i\cdot\prod^{n+1}_{j=1}\zeta_j\bigg)\nonumber&\\
		&=\frac{1}{2(n+1)}\bigg[(n+1)2^n\cdot2^{n+r+1}\cdot3^{n-r}\bigg]\nonumber&\\
		&=2^{2n+r}\cdot3^{n-r}.&
	\end{flalign*} 
	
	Therefore, we have
	\begin{flalign*}
		&\hspace{1cm}\tau(G^r_{n})=2^{2n+r+2d-10}\cdot3^{n-r-d+5},&
	\end{flalign*}
	as desired.
\end{proof}

\section*{Acknowledgements} The authors would like to express their sincere gratitude to all the referees for their careful reading and insightful suggestions.





\end{document}